\documentclass[reqno, 12pt]{amsart}

\usepackage{mathrsfs}
\usepackage{amscd}
\usepackage{amsmath}
\usepackage{latexsym}
\usepackage{amsfonts}
\usepackage{amssymb}
\usepackage{amsthm}
\usepackage{graphicx}
\usepackage{hyperref}
\usepackage{makecell}
\usepackage{array}
\usepackage{booktabs}
\usepackage{multirow}
\usepackage{color,xcolor}

\newtheorem{theorem}{Theorem}[section]

\newtheorem{remark}[theorem]{Remark}
\newtheorem{Thm}{Theorem}[section]
\newtheorem{Def}{Definition}[section]
\newtheorem{Le}{Lemma}[section]

\newtheorem{Rem}{Remark}[section]

\numberwithin{equation}{section}

\title[Inverse potential problem]{Stability for the inverse random potential scattering problem}

  {\author[T. Wang]{Tianjiao Wang}\address{School of Mathematical Sciences, Zhejiang University, Hangzhou 310058, China}}\email{wangtianjiao@zju.edu.cn}

\author[X. Xu]{Xiang Xu}
\address{School of Mathematical Sciences, Zhejiang University, Hangzhou 310058, China}
\email{xxu@zju.edu.cn}
\author[Y. Zhao]{Yue Zhao}
\address{School of Mathematics and Statistics, and Key Lab NAA-MOE, Central China Normal University,
Wuhan 430079, China}
\email{zhaoyueccnu@163.com}


\begin{document}

\begin{abstract}
This paper is concerned with an inverse random potential problem for the Schr\"odinger equation. The random potential is assumed to be
a generalized Gaussian random function, whose covariance operator is a classical pseudo-differential operator. 
For the direct problem, the meromorphic continuation of the resolvent of the Schr\"odinger operator with rough potentials is investigated,
which yields the well-posedness of the direct scattering problem and a Born series expansion. For the inverse problem,
we derive a probabilistic stability
estimate for determining the principle symbol of the covariance operator of the random potential. The stability result provides an estimate 
of the probability for an event when the principle symbol can be quantitatively determined by a single realization of the multi-frequency backscattered far-field pattern.
The analysis employs the ergodicity theory and quantitative analytic continuation principle.

\end{abstract}

\maketitle

\section{Introduction}

Stochastic Schr\"odinger equation is a fundamental mathematical formulation in quantum mechanics under the influence of random processes. 
It can also be used as the paraxial wave equation in random media which is a model for many applications, for instance in communication and imaging
\cite{GS, GS-MMS, GSARMA, GPT}.
In such stochastic models, the potential function is usually assumed to be a random process. In stochastic inverse scattering theory, we aim to determine the statistical properties such as mean and variance of the random potential from appropriate measurements away from its support. 
Inverse scattering problems have wide applications in many scientific and industrial areas such as geophysical exploration, biomedical imaging, and radar and sonar \cite{Colton}.

The deterministic inverse potential problems for the Schr\"odinger equation have been extensively investigated. In \cite{Uhlmann}, Uhlmann and Sylvester solved the famous Calder$\acute{\rm o}$n problem by transforming it into an inverse potential problem for the  Schr\"odinger equation. The method of constructing complex geometric optics (CGO) solutions was developed in this work which is now a fundamental tool in studying uniqueness and stability for inverse potential problems \cite{Ale, Isakov1, U_1}. Stability estimates for the corresponding inverse scattering problems were also obtained by constructing CGO solutions \cite{HH, IN}. However, this method has a limitation, that is, it not applicable to derive stability in two dimensions. To break this limit, by using multi-frequency data, a novel methodology has been  developed in \cite{ZZ} recently, which achieves the stability for the inverse potential scattering problem in two dimensions without using CGO solutions.

Compared with the deterministic cases, inverse random potential problems are much less studied. In \cite{Caro, LPS}, the authors studied the unique determination of an inverse random potential problem for 
the Schr\"odinger equation. In these works, the potential was assumed to be a generalized microlocally isotropic Gaussian random process whose covariance operator is a classical 
pseudo-differential operator. Given the multi-frequency random data, the principal symbol of the covariance operator was demonstrated to be uniquely determined by a single realization of the multi-frequency data averaged over the high frequency band with probability one by ergodicity. However, the stability for this inverse random potential problem remains wide open. In \cite{wang2024stability}, an increasing stability for recovering the random source from the expectation of the multi-frequency correlation data was proved by applying the quantitative analytic continuation principle. However, the analysis for the inverse source problem is not applicable to the 
inverse random potential problem. This is because the random potential can not be uniformly bounded with respect to different realizations. As a result,
the desired analyticity of the data may not be obtained.

In this paper, we aim to establish a methodology for deriving a probabilistic stability estimate for the inverse random potential scattering problem for the Schr\"odinger equation. The potential is assumed to be a generalized Gaussian random field as proposed in \cite{LPS}. Such random field includes classical stochastic processes such as white noise, fractional Brownian motions, Markov process and so on.  However, the method developed by \cite{Caro, LPS} for proving the uniqueness may not be directly applicable to obtain a quantitative 
stability estimate. An essential difficulty could be that the ergodicity theorem used in these works does not yield a uniform estimate for the difference between the expectation of the correlated zeroth-order
term of the Born series and that at a single realization. Thus, one may not obtain a stability which almost surely holds.
To deal with this issue, we consider an event when it is possible to obtain a quantitative stability using the multi-frequency far-field pattern at a single realization. In this step, we use the analytic
continuation principle which requires the analyticity and estimates of the resolvent.
On the other hand, the probability of the corresponding complementary event  is estimated by the Chebyshev inequality.
As a consequence, we establish a probabilistic stability which gives an estimate of the probability of an event when the strength of the random potential 
can be quantitatively determined by the multi-frequency far-field pattern at a single realization. Moreover, the probability approaches one as the frequency increases.
It should be mentioned that with the aid of analyticity of the resolvent, we only use data measured in a finite interval of frequencies, while \cite{Caro, LPS} require data at all high frequencies.

This paper is organized as follows. In Section \ref{sec2} we describe the stochastic Schr\"odinger equation and formulate the main result of this paper.
The resolvent of the Schr\"odinger operator with rough potentials is investigated in Section \ref{re}. Based on the analysis of the resolvent, well-posedness of the direct
scattering problem is established together with a Born series expansion. Section \ref{IP} is devoted to the proof of the main result. 
The zeroth, first, and higher order terms of the Born series expansion are investigated. 
The analysis in this section relies on an ergodicity result and applying the Chebyshev inequality and analytic continuation principle.

\section{Problem formulation and main result}\label{sec2}

We formulate the stochastic scattering problem. Let $u^{inc}:=e^{{\rm i}k\theta \cdot x}$  be an incident plane wave with incident direction $\theta \in \mathbb{S}^2$. Denote the scattered wave and the total wave by $u$ and $u^{sc}$, respectively. Then the scattering problem can be formulated as \begin{align}\label{DSP}
	\begin{cases}
		(\Delta+k^2+V)u=0 ,& \quad \text{in}\quad \mathbb{R}^3,  \\
		u=u^{inc}+u^{sc}, & \quad \text{in}\quad \mathbb{R}^3,\\
		\lim\limits_{r \to \infty} r (\partial_r u- ik u)=0, &\quad r=|x|,
	\end{cases} 
\end{align}
where $k>0$ is the wavenumber\slash frequency, and $V$ is the potential function supported in a bounded and simply connected domain $D$.

Letting $(\Omega, \mathcal{A}, \mathbb{P})$ be a complete probability space,
we assume that the potential $V$ is a real valued generalized Gaussian function supported in $D$. This means that $V$ is a measurable map from the
sample space $\Omega$ to the space of real-valued distributions $\mathcal D'(\mathbb R^3)$ such that $\omega \mapsto \langle g(\omega), \phi \rangle$ is a real-valued Gaussian random variable for all $\phi \in C_0^\infty(\mathbb R^3)$. The expectation of the random function $V$ is a generalized function given by \[
\mathbb{E}V\, :\, \phi \mapsto \mathbb{E}\langle V,\phi \rangle, \quad \phi \in C_0^\infty(\mathbb R^3),
\] and the covariance is defined as a bilinear form by \[
\mathrm{Cov}V\,:\, (\phi_1, \phi_2) \mapsto \mathrm{Cov}(\langle g,\phi_1 \rangle, \langle V,\phi_2 \rangle), \quad \phi_1, \phi_2 \in C_0^\infty(\mathbb R^3),
\] which gives the covariance operator $C_V$ by \begin{equation}\label{2.1}
	\langle C_V \phi_1, \phi_2 \rangle = \mathrm{Cov}(\langle V,\phi_1 \rangle, \langle V,\phi_2 \rangle), \quad \phi_1, \phi_2 \in C_0^\infty(\mathbb R^3).
\end{equation} Denote the Schwarz kernel of $C_V$ by $K_V(x,y)$, i.e., \begin{equation}\label{2.2}
	\langle C_V \phi_1, \phi_2 \rangle=\int_{\mathbb{R}^3}\int_{\mathbb{R}^3} K_V(x,y) \phi_1(x)\phi_2(y) \,\mathrm{d}x\mathrm{d}y. \end{equation} Combining 
\eqref{2.1}--\eqref{2.2} yields \[
K_V(x,y)=\mathbb{E}[(V(x)-\mathbb{E}V(x))(V(y)-\mathbb{E}V(y))].
\]

In this paper, we focus on the following class of generalized Gaussian random functions proposed by \cite{LPS}, which are
called the Gaussian microlocally isotropic random function.
\begin{Def}\label{Def2.1}
	A generalized Gaussian random function $g$ with zero expectation is called \textbf{microlocally isotropic of order} $-m \in \mathbb{R}$ in the domain $D \subset \mathbb{R}^3$, if $\text{supp}\, g \subset \subset D$ for almost surely $\omega \in \Omega$ and its covariance operator $C_g$ is a classical pseudo-differential operator with the principal symbol $h(x)|\xi|^{-m}$ where $0 \le h \in C^\infty_0(\mathbb{R}^3)$ and $\text{supp} \, h \subset \subset D$. The function $h$ is called the micro-correlation strength of $g$.
\end{Def} 

Based on the above definition, assuming the random potential $V$ is a generalized microlocally isotropic Gaussian (GMIG) random field of order $-m$ in $D$ and
letting $c_V(x,\xi)$ be the symbol of $C_V$, we have \[
C_V(\phi)(x)=(2 \pi)^{-3} \int_{\mathbb{R}^3} e^{i x \cdot \xi} c_V(x, \xi) \hat{\phi}(\xi) \,\mathrm{d}\xi, \quad \phi \in C_0^\infty(\mathbb{R}^3),
\] where $\hat{\phi}(\xi)$ stands for the Fourier transform of $\phi$ defined by $
\hat{\phi}(\xi)=\mathcal{F}\phi(\xi):=\int_{\mathbb{R}^3} e^{-i x \cdot \xi} \phi(x) \,\mathrm{d}x.$  
The kernel $K_V$ can be represented as an oscillatory integral of the form \begin{equation*}
	K_V(x,y)=(2\pi)^{-3} \int_{\mathbb{R}^3} e^{i (x-y) \cdot \xi} c_V(x, \xi) \,\mathrm{d}\xi.
\end{equation*}
\begin{remark}
Notice that according to \cite{li2021inverse}, $V \in W^{\frac{m-3}{2}-\delta,p}(D)$ almost surely for all $\delta>0$ and $1<p<\infty$. Thus, in particular, we are interested 
in the cases when $m\leq 3$, which correspond to rough fields. For the inverse problem in this paper, we assume that the random potential $V$ is a real-valued microlocally isotropic Gaussian random function of order $-m$, where $m$ satisfies $m \in (14/5,4)$, which includes such cases.
\end{remark}
Now we formulate the inverse problem considered in this work. 
Assume that $V$ is a microlocally isotropic generalized Gaussian random function of order $-m$ with micro-correlation strength $h$ in a bounded Lipschitz domain
$D \subset \mathbb{R}^3$.  The scattered wave field $u^{sc}$ has the following asymptotic expansion as $|x|\to\infty$,
\[
u^{sc}(x) = \frac{e^{ik|x}}{|x|} u^\infty(\hat{x},k, \theta) + O\Big(\frac{1}{|x|^2}\Big),
\]
where $u^\infty(\hat{x}, k, \theta)$ is referred to as the far-field pattern associated with the scattering problem with $\hat{x}= \frac{x}{|x|}$ being the observation direction. 
Given an interval $I = (K_0, K]$ of multiple frequencies, we are interested in the following inverse potential problem:. 

{\it Inverse Random Potential Problem:} Determine the micro-correlation strength $h$ from the multi-frequency far-field pattern $\{u^\infty(\hat{x},k, \theta): \hat{x}, \theta\in\mathbb S^2, \,\, k\in I\}$.

Now we present the main stability result in this paper.
Assume that $V_1$ and $V_2$ are two random potentials with micro-correlation strengths $h_1$ and $h_2$, respectively. 
Denote the far-field data corresponding to $V_1$ and $V_2$ by
$u^\infty_{(1)}(\hat{x},\theta,k)$ and $u^\infty_{(2)}(\hat{x},\theta,k)$, respectively. 
Introduce the backscattered far-field data discrepancy in a finite interval $I = (K_0, K]$ with $0<K_0<K$ as follows 
\[
\varepsilon^2 =  \sup_{k\in I, \tau\in(0, 1/2)} \frac{1}{k}\int_k^{2k}\int_{\mathbb S^2}|s^{m}U(s,\theta)|^2\,\mathrm{d}\sigma(\theta)\,\mathrm{d}s, 
\] where \[U(s,\theta)=\overline{u_{(1)}^\infty(-\theta, \theta,s+\tau)} {u_{(1)}^\infty(-\theta;\theta,  s)}-\overline{u_{(2)}^\infty(-\theta, \theta,s+\tau)} {u_{(2)}^\infty(-\theta;\theta,  s)}.\]

The goal of this work is to derive a stability estimate for the inverse random potential problem. To this end,
denoting $a_V(x,\xi):=c_V(x,\xi)-h(x)|\xi|^{-m}$, we need a priori assumptions on the potential function. Let $\mathcal V$ be the set
which consists of GMIG random fields of order $-m$ satisfying the following a priori estimates: 
\begin{align*}
|h(x)| \le M_1 , \|c_V(x,\xi)\|_{C^1_x(D)} \le M_1(1+|\xi|)^{-m}, |a_V(x, \xi)1_{|\xi|\ge 1}| \le M_1|\xi|^{-(m+1)}
\end{align*}
for all $x\in D$ and $\xi\in\mathbb R^3$. Here $M_1, M_2>0$ are two positive constants.
Let $\mathcal C_r:=\{h:\|h\|_{H^r} \le M_2\}$ with $r>0$ being a constant. When $m \in (2,3]$, denote the index $\alpha$ by $
		\alpha=
			\frac{m-3}{2}-\delta
	$ with a sufficiently small positive constant $\delta$ and  let $\alpha=0$ when $m  \in (3,4).$

The following stability is the main result of this paper. Hereafter, the notation $a \lesssim b $ stands for $a \le C b$, where $C>0$ is a generic constant whose special value is not required but should be clear from the context.  
\begin{Thm}\label{Thm5.1} Let $V_1,\,V_2 \in \mathcal V$ be two random potentials with $h_1,h_2 \in \mathcal C_r$ and $\epsilon\in (0, 1/4)$. Furthermore, assume almost surely  $\mathbb E\|V_j\|_{W^{\alpha,p}} \le M_3$ with $p \in [ -\frac{3}{\alpha}, +\infty)$ for $\frac{14}{5}<m<3$, and $\mathbb E\|V_j\|_{L^\infty} \le M_3$ for $3 \le m <4$, where $j=1, 2$ and $M_3$ is a positive constant.
	For any $\alpha \in (0,1/4)$ and $M_0>0$, there exist constant $\beta_1,\beta_2>0$ such that
	\begin{align}\label{PIstab}
		\mathbb P \left(\|h_1-h_2\|_{L^2(\mathbb{R}^3)}^2
		\lesssim \left|\ln \left( C\varepsilon^2+\frac{M_0+K_0^{3+4\alpha}}{(K\ln|\ln \varepsilon|)^{\beta_1}} \right) \right|^{-\beta_2}\right)\ge 1-\frac{1}{M_0}-\frac{1}{K^{a(\epsilon)}}-\frac{1}{\sqrt K_0},
	\end{align} 
	where the positive constant $a(\epsilon) \in (0,1)$ depends on $\epsilon $, and the positive constants $\beta_1,\beta_2$ depend on $r$, $m$, $\alpha$ and $\epsilon.$ 
  \end{Thm}
\begin{Rem}
The estimate \eqref{PIstab} is a probabilistic estimate which shows the probability of the even when the strength of the random potential can
be quantitatively determined by the multi-frequency far-field pattern at a single realization. Moreover, it indicates that as the frequencies of the data and the arbitrarily chosen
constant $M_0$ increase 
the probability of such event increases and approaches one.
\end{Rem}

\section{Direct scattering problem}\label{re}

In this section, we investigates the meromorphic continuation of the resolvent for the Schr\"odinger operator with rough potential.
An analytic region and resolvent estimates with respect to complex wavenumber are obtained, which are useful in the derivation of the stability.  Moreover, as a result, we establish the well-posedness of the direct scattering problem and a Born series expansion of the far-field pattern.  

We introduce some function spaces. The function spaces $W^{s,p}:=W^{s,p}(\mathbb{R}^3)$ denote the standard Sobolev spaces and denote $H^s:=W^{s,2}$. Furthermore, the function spaces $H^s_{loc}:=H^s_{loc}(\mathbb{R}^3)$, $W^{s,p}_{comp}:=H^s_{comp}(\mathbb{R}^3)$ are respectively defined by \begin{align*}
	H^s_{loc}(\mathbb{R}^3) &:=\{u \in \mathcal{D}'(\mathbb{R}^3):\chi u \in H^s(\mathbb{R}^3),\forall \, \chi \in C_0^\infty(\mathbb{R}^3)\}, \\ H^s_{comp}(\mathbb{R}^3) &:=\{u \in H^s(\mathbb{R}^3):\exists \, \chi \in C_0^\infty(\mathbb{R}^3),\,\chi u =u\},
\end{align*} where $\mathcal{D}'$ is the generalized function space. The bounded linear operators between the function spaces $X$ and $Y$ are denoted by $\mathcal{L}(X,Y)$ with the operator norm $\|\cdot\|_{\mathcal{L}(X,Y)}$. 

Denote the free resolvent of the Schr\"odinger by ${R}_0(\lambda):=(-\Delta-\lambda^2)^{-1}$ which has the explicit kernel 
\begin{align} 
	\Phi_\lambda(x,y)=
	\frac{e^{i\lambda|x-y|}}{4\pi |x-y| }.
	\label{4.2}
\end{align}  The following lemma concerns the analyticity and estimates of the free resolvent. The proof adapts the arguments in 
\cite{dyatlov2019mathematical}.
\begin{Le}\label{Lem4.1}
	The free resolvent $R_0(\lambda)$ is analytic for $\lambda \in \mathbb{C}$ with $\Im \lambda>0$ as a family of operators \[
	R_0(\lambda): H^s(\mathbb{R}^3) \to H^s(\mathbb{R}^3)
	\] for all $s \in \mathbb{R}$. Furthermore, given a cut-off function $\chi \in C_0^\infty(\mathbb{R}^3)$, ${R}_0(\lambda)$ can be extended to a family of analytic operators for $\lambda \in \mathbb{C}$ as follows \[
	\chi R_0(\lambda) \chi : H^s(\mathbb{R}^3) \to H^s(\mathbb{R}^3)
	\] with the resolvent estimates 
	\begin{equation*}
		\|\chi R_0(\lambda) \chi \|_{\mathcal{L}(H^s,H^t)} \lesssim (1+|\lambda|)^{t-s-1}e^{L (\Im \lambda)_-},
	\end{equation*} where $t\in [s,s+2]$, $v_-:=\max\{-v,0\}$ and $L$ is a constant larger than the diameter of the support of $\chi$.
\end{Le}
\begin{proof}
	With the help of the functional calculus \cite[Appendix B.2]{dyatlov2019mathematical}, we can denote \[
	U(t):=\frac{\sin t\sqrt{-\Delta}}{\sqrt{-\Delta}}, 
	\] which gives a representation of $R_0$ as follows
	\[
	R_0(\lambda)=\int_0^\infty e^{i\lambda t} U(t)\,\mathrm{d}t.
	\] The strong Huygens's principle implies that \[
	(U(t) \chi)(x)=0,\quad t>\sup\{|x-y|:y \in \text{supp} \chi\}.
	\] Thus we obtain \[
	\chi  R_0(\lambda) \chi =\int_0^L e^{i\lambda t} \chi U(t) \chi \,\mathrm{d}t,
	\] where $L> \text{diam}\,\text{supp} \chi$.
	The spectral representation \cite[Appendix B.1]{dyatlov2019mathematical} directly
	gives \[
	\partial_t^k U(t) \,:\, H^s(\mathbb{R}^3) \to H^{s-k+1}(\mathbb{R}^3),\quad k \in \mathbb{N},\quad s \in \mathbb{R}.
	\]  Furthermore, we have \[
	\chi  R_0(\lambda) \chi=(i\lambda)^{-1}\int_0^L  \partial_t(e^{i\lambda t}) \chi U(t) \chi \,\mathrm{d}t=-(i\lambda)^{-1}\int_0^L  e^{i\lambda t} \chi  \partial_t(U(t)) \chi \,\mathrm{d}t.
	\] Since $\partial_tU(t)=\cos{\sqrt{-\Delta}}=O_{\mathcal{L}(H^s, H^s)}(1)$ by the spectral representation, we prove the lemma for the case $s=t$. 
	The remaining cases $s<t \le t+2$ follow by applying the standard elliptic regularity theory and the interpolation inequality.
\end{proof}

 Next, we shall prove resolvent estimates for the resolvent $R_V(\lambda):=(-\Delta-\lambda^2-V)^{-1}$ of the Schr\"odinger operator with rough potentials. The following lemma in \cite{Caro} is useful in the analysis.  \begin{Le}\label{Lem5.1}
	Assume that $V \in W^{s,p}_{comp}(\mathbb{R}^3)$ with $s \in [-\frac{1}{2},0)$ and $p \in [-\frac{3}{s},  +\infty)$. We have \[
	\|V f\|_{H^{s}(\mathbb{R}^3)} \lesssim \| V\|_{W^{s,p}(\mathbb{R}^3)} \| \chi f\|_{H^{-s}(\mathbb{R}^3)},\quad \forall\, f \in H^{-s}_{loc}(\mathbb{R}^3), 
	\] where $\chi \in C_0^\infty(\mathbb{R}^3)$ is a cut-off function with $\chi|_{D}=1$. 
\end{Le}
Applying Lemma \ref{Lem5.1} and the perturbation arguments in \cite[Theorem 3.10]{dyatlov2019mathematical}, we derive the following resolvent estimates for rough potentials.
\begin{Le}\label{Lem5.2}
	 Assume $V \in W^{s,p}_{comp}(\mathbb{R}^3)$ with $s \in [-\frac{1}{2},0)$ and $p \in [-\frac{3}{s},  +\infty)$. Then the resolvent $R_{ V}(\lambda):=(-\Delta-\lambda^2-V)^{-1}$ is meromorphic for $\lambda \in \mathbb{C}$ with $\Im \lambda>0$ as a family of operators $R_{ V}(\lambda): H^s \to H^t$ with $t \in [s,s+2]$, which can be extend to a family of meromorphic operators for \[
      R_{ V}(\lambda):=H^s_{comp} \to H^t_{loc},\quad \lambda \in \mathbb C
      \] with $t \in [s,s+2]$. Further, fix a cut-off function $\chi \in C_0^\infty(\mathbb{R}^3)$. Then there exist constants $A,C_0,\delta,C$ and $L$ such that the estimates \begin{equation*}
		\|\chi R_V(\lambda) \chi \|_{H^s \to H^t} \le C (1+|\lambda|)^{t-s-1}e^{L (\Im \lambda)_-}
	\end{equation*} hold for $t \in [s,s+2]$ and \[
	\lambda \in \mathscr{S}:=\{\lambda \in \mathbb{C}: \Im \lambda \ge -A-\delta \log(1+|\lambda|), |\lambda|>C_0\| V\|^2_{W^{s,p}(\mathbb{R}^3)}\}.
	\]  Moreover, the resolvent $\chi R_{ V}(\lambda) \chi: H^s\to H^s$ is analytic for $\lambda \in \mathscr{S}$.  Here the constants $A,\delta,L$ depend on $\chi$ and $supp\, V$, and the constants $C, C_0$ depend on $\chi$.
\end{Le}
\begin{proof}
 Formally, we have the resolvent identity 
 \[
 R_{ V}(\lambda) = R_0(\lambda) (I  + VR_0(\lambda))^{-1}.
 \]
 Using Lemma \ref{Lem5.1}, it can be verified when $\Im \lambda>0$ \begin{align*}
 \| V  R_0(\lambda) \|_{\mathcal{L}(H^{s},H^{s})} &\lesssim \|V\|_{W^{s,p}}\| R_0(\lambda) \|_{\mathcal{L}(H^{s},H^{-s})}  \lesssim (1+|\lambda|)^{-1-2s}.
 \end{align*} Hence, when $\Im \lambda>0$ and $|\lambda| >>1$, the operator $I   + VR_0(\lambda) : H^s \to H^s$ is invertible by the following Neumann series argument \[
 (I +  VR_0(\lambda))^{-1}=\sum_{j=0}^\infty (-VR_0(\lambda))^j.
 \] Moreover, from Lemma \ref{Lem4.1}--\ref{Lem5.1} we have that $ VR_0(\lambda) : H^s \to H^s$ is a compact operator. Therefore, by applying the analytic Fredholm theory \cite[Theorem C.8]{dyatlov2019mathematical} we have that the resolvent \[
 R_{ V}(\lambda) = R_0(\lambda) (I +   VR_0(\lambda))^{-1} : H^s \to H^t
 \] is meromorphic with respect to $\lambda$ with $\Im \lambda>0$ for $t \in [s,s+2]$.
 
 Next we consider the operator $\chi R_V(\lambda)\chi$. There holds the resolvent identity \begin{align}\label{RI}
     \chi R_{ V}(\lambda)\chi = \chi R_0(\lambda) \chi_1 (I +   VR_0(\lambda) \chi_1)^{-1}(1 + VR_0(\lambda)(1-\chi_1))  \chi,
 \end{align} where $ \chi_1\in C_0^\infty(\mathbb R^3)$ such that $\chi_1 = 1$ on $\text{supp}\,V$. By applying Lemma \ref{Lem4.1}--\ref{Lem5.1} and the analytic Fredholm theory again, we obtain that $R_{V}: H^s_{comp} \to H^t_{loc}$ can be extended to a family of meromorphic operators for $\lambda \in \mathbb C$ with $t \in [s,s+2]$. And we have \begin{align*}
    \| V R_0(\lambda) \chi_1  \|_{\mathcal{L}(H^{s},H^{s})} \lesssim \| V\|_{W^{s,p}(\mathbb R^3)} (1+|\lambda|)^{-1-2s}e^{L (\Im \lambda)_-} ,\quad \lambda\in\mathscr{S}.
\end{align*}
We see that for $\delta$ sufficiently small, say $\delta<\frac{1}{2L(1+2s)}$ and $C_0$ sufficiently large, 
\[
\| V R_0(\lambda)\chi_1\|_{\mathcal{L}(H^{s},H^{s})} \lesssim \| V\|_{W^{s,p}(\mathbb R^3)}|\lambda|^{-1/2}<1.
\]
Thus, from a Neumann series argument we have $(I +  VR_0(\lambda)\chi)^{-1}: H^s_{comp}\to H^s_{comp}$
is well-defined. From the resolvent identity \eqref{RI}, the free resolvent estimates in Lemma \ref{Lem4.1} and the analytic Fredholm theory, we have that $\chi R_{ V}(\lambda) \chi:  H^s\to H^t$ is analytic for $\lambda\in\mathscr S$
with the resolvent estimate
\begin{align*}
     &\|\chi R_{ V}(\lambda) \chi \|_{\mathcal{L}(H^s,H^t)}  \\ &\le \|\chi R_0(\lambda) \chi_1 \|_{\mathcal{L}(H^s,H^t)}\frac{1}{1-\|VR_0(\lambda)\chi_1\|_{\mathcal L(H^s,H^s)}}\|(1 + VR_0(\lambda)(1-\chi_1))\chi\|_{\mathcal L(H^s,H^s)}  \\ &\lesssim (1+|\lambda|)^{t-s-1}e^{L (\Im \lambda)_-} \frac{1}{1-C\|V\|_{W^{s,p}(\mathbb R^3)}|\lambda|^{-1/2}} \\ &\lesssim(1+|\lambda|)^{t-s-1}e^{L (\Im \lambda)_-},\quad t \in [s,s+2],
 \end{align*} where $|\lambda|>C_0\|V\|^2_{W^{s,p}(\mathbb R^3)}$ for sufficiently large $C_0.$
 The proof is complete.
 \end{proof}
 
Based on the above arguments, we establish the well-posedness of the direct scattering problem \eqref{DSP} for large wavenumbers.
 
\begin{Thm}\label{wp}
    Assume that the potential $V$ is a GMIG random field of order $-m$ in $D $ with $m>2$. For $k$ being sufficiently large, the scattering problem \eqref{DSP} almost surely admits a unique solution $u^{sc} \in H^{\alpha+2}_{loc}$ with non-positive constant $\alpha$ satisfying $-1/2<\alpha<(m-3)/2$.
\end{Thm}
\begin{proof} 
From \cite{li2021inverse} we have $V \in W^{\alpha,p}(D)$ with non-positive $\alpha$ satisfying $-1/2<\alpha<(m-3)/2$ and $p \in [-\frac{3}{\alpha},+\infty)$. 
Taking $\chi$ and $\chi_1$ 
such that $\chi =1$ in $D $ and $\chi_1 =1$ in ${\rm supp}\, \chi$, we can rewrite \eqref{RI} as \begin{equation}\label{RI2}
 R_{ V}(k)\chi =  R_0(k) \chi_1 (I + VR_0(k)\chi_1)^{-1}\chi.
\end{equation}  Then combining \eqref{RI2} and Lemma \ref{Lem4.1}-- \ref{Lem5.2} we obtain for sufficiently large $k>0$ that \[
 u^{sc}:= R_{V}(k) Vu^{inc}=R_{ V}(k)\chi Vu^{inc}
\] is well defined and $u^{sc} \in H^{\alpha+2}_{loc}$, which is a distributional solution to the scattering problem \eqref{DSP}. 

In what follows we prove the uniqueness. 
Assume that $\tilde u^{sc} \in H^{\alpha+2}_{loc}$ is a solution to the direct scattering problem \eqref{DSP} in the sense of distribution. Next we show that $\tilde u^{sc} $ satisfies the Lippmann-Schwinger equation \begin{align}\label{LSEQ}
    \tilde u^{sc}(x)=-\int_{\mathbb R^3}\Phi_k(x,y) V(y)(e^{{\rm i}kd \cdot y}+\tilde u^{sc}(y))\,\mathrm{d}y.
\end{align}
To this end,  
choose a sequence  $\{\phi_n\}_{n=1}^\infty \subset C^\infty(\mathbb R^3)$ such that $\phi_n \to \tilde u^{sc} $ in $H^{\alpha+2}_{loc}$. Then using integration by parts gives \begin{align}\label{r1}
   &\int_{B_R} (\Delta+k^2) \phi_n(y) G(x,y,k)\,\mathrm{d}y \notag\\
   &=\phi_n(x)+\int_{\partial B_R}\left(\partial_{\nu(y)} \phi_n(y) \Phi_k(x,y) - \phi_n(y)\partial_{\nu(y)} \Phi_k(x,y)\right){\rm d}s\notag\\
 &\quad + \int_{\partial B_R} \left(\partial_{\nu (y)}\phi_n(y)  \Phi_k(x,y) - \phi_n(y)\partial_{\nu(y)}  \Phi_k(x,y)\right) {\rm d}s \quad \text{for} \,x \in B_R.
\end{align}
From Sobolev embedding theorem we have $\tilde u^{sc} \in H^{\alpha+2}(B_R) \subset C(B_R)$ for $\alpha>-1/2$, which implies $|\phi_n(x) \to \tilde u^{sc}(x)| 
\to 0$ for any $x \in B_R$. 
Noticing $(\Delta+k^2)\tilde u^{sc} \in H_{comp}^\alpha (B_R)$, by applying \cite[Lemma 3.2]{wang2024stability} and Sobolev embedding theorem we obtain $\Phi_k(x,\cdot) \in H^{-\alpha}(B_R)$ as $\alpha>-\frac{1}{2}$, which gives \[
\lim_{n\to\infty}\int_{B_R} (\Delta+k^2) \phi_n(y) \Phi_k(x,y)\,\mathrm{d}y=\int_{B_R} (\Delta
+k^2) \tilde u^{sc}(y) \Phi_k(x,y)\,\mathrm{d}y. 
\] 
 Therefore, by letting $n \to \infty$ in \eqref{r1} we arrive at \begin{align*}
    \tilde u^{sc}(x) &=\int_{B_R} (\Delta+k^2) \tilde u^{sc}(y) \Phi_k(x,y)\,\mathrm{d}y\\
    &\quad+\int_{\partial B_R}\left(\partial_{\nu(y)}  \tilde u^{sc}(y) \Phi_k(x,y)-\tilde u^{sc}(y)\partial_{\nu(y)} \Phi_k(x,y)\right){\rm d}s\\
 &\quad + \int_{\partial B_R} \left(\partial_{\nu (y)}\tilde u^{sc}(y)  \Phi_k(x,y) - \tilde u^{sc}(y)\partial_{\nu(y)}  \Phi_k(x,y)\right)
 {\rm d}s,\quad x \in B_R,
\end{align*} 
which yields the Lippmann-Schwinger equation \eqref{LSEQ}
by taking $R \to \infty$ and using the radiation condition.

At last, we show that when $k$ is sufficient large, the Lippmann-Schwinger equation admits a unique solution, which ensures the uniqueness of the scattering problem. 
It suffices to verify that the homogeneous equation \begin{align*}
     v(x)=-\int_{\mathbb R^3}\Phi_k(x,y)V(y)v(y)\,\mathrm{d}y.
\end{align*} only admits the trivial solution in $H^{-\alpha}_{loc}$. The above equation can be rewritten as 
\[ v= R_0(k) \chi V v. \]
Then applying Lemma \ref{Lem4.1}--\ref{Lem5.1} gives \begin{align*}
    \|\chi v\|_{H^{-\alpha}} \le \frac{C(\chi)}{k^{2\alpha+1}}\|\chi v\|_{H^{-\alpha}},
\end{align*} which implies $\chi v \equiv 0$ as $k$ is large enough. Thus, noting $\chi=1$ in $D$, we have $v=0 $ in $D$. Since $v$ satisfies the Helmholtz equation 
\[
(\Delta+k^2)v=0 \quad \text{in} \quad \mathbb R^3 \backslash \overline{D},
\]
and $ \text{supp}\,V\subset\subset D$,
 the unique continuation for elliptic equations implies $v \equiv 0$ in $\mathbb R^3$.
The proof is complete.
\end{proof}

Based on the above analysis, the scattered field $u^{sc}$ admits a Born
series expansion, which is useful for the subsequent study of the inverse problem. Indeed, we have \begin{align}\label{bs}
u^{sc}: &=R_{ V}(k)( Vu^{inc}) = R_{ V}(k)\chi(Vu^{inc})\notag\\ &=R_0(k) \chi_1 (I  + VR_0(k)\chi_1)^{-1}\chi( Vu^{inc}) \notag\\ &= -R_0(k)\sum_{j=0}^\infty(-VR_0(k))^j ( Vu^{inc}),
\end{align} 
and from Lemma \ref{Lem4.1}--\ref{Lem5.1} the series \begin{align*}
    \sum_{j=0}^\infty(-  VR_0(k))^j (  Vu^{inc})=\sum_{j=0}^\infty(-  V\chi R_0(k)\chi)^j (Vu^{inc})
\end{align*} converges in $H^\alpha$ for sufficiently large $k$.

\section{Stability for the inverse random potential problem}\label{IP}
In this section, we prove the stability estimate in Theorem \ref{Thm5.1}. We begin by studying the Born series expansion.
For convenience, 
in the following analysis, we drop the subscripts of $V_j$ and $u^\infty_{(j)}$. 
The Born series expansion \eqref{bs} gives \begin{align}\label{5.2.2}
	u^\infty(\hat{x},\theta,k)&= \frac{1}{4\pi}  \int_{\mathbb{R}^3} e^{-ik\hat{x} \cdot y}V(y)\sum_{j\ge 0}[(V R_0(k) )^{j} u^{inc}](y)\,\mathrm{d}y \notag\\ &=\frac{1}{4\pi}  \int_{\mathbb{R}^3} e^{-ik(\hat{x}-\theta) \cdot y} V(y)\,\mathrm{d}y+\frac{1}{4\pi}  \int_{\mathbb{R}^3} e^{-ik\hat{x} \cdot y}V(y)[R_0(k)V u^{inc}](y)\,\mathrm{d}y \notag\\ &\quad +\frac{1}{4\pi} \int_{\mathbb{R}^3} e^{-ik\hat{x} \cdot y}\sum^\infty_{j=2}[(R_0(k) V)^{j} u^{inc}](y)\,\mathrm{d}y \notag\\ &:=u_0^\infty(\hat{x};\theta,k) + u_1^\infty(\hat{x};\theta,k) + u_2^\infty(\hat{x};\theta,k).
\end{align} 
\subsection{Analysis of the correlation of the zeroth order term in Born series}
In this section, we study the correlation of the zeroth order term in Born series expansion. A novel ergodicity result is proved, which is important in the derivation of the stability.

Letting $\hat{x}=-\theta$, which means only the backscattering data is utilized, from direct calculations we obtain \begin{align*}
	&\mathbb E[u_0^\infty(-\theta,\theta, k) \overline{u_0^\infty(-\theta,\theta, k+\tau)}]\\
	&=\frac{1}{16\pi^2}\int_{\mathbb R^3}\int_{\mathbb R^3} e^{-2{\rm i}(k+\tau)\theta \cdot y}  e^{2{\rm i}k\theta \cdot z}
	\mathbb E[V(y)V(z)]{\rm d}y{\rm d}z\\
	&= \frac{1}{16\pi^2}\int_{\mathbb R^3} \Big[\int_{\mathbb R^3} K_V(y, z)e^{2{\rm i}k\theta\cdot z}{\rm d}z\Big]
	e^{-2{\rm i}(k+\tau)\theta \cdot y}{\rm d}y \\
	&= \frac{1}{16\pi^2} \Big[\int_{\mathbb R^3} h(y) e^{-2{\rm i}\tau\theta\cdot y}{\rm d}y |2k\theta|^{-m} + \int_{D} a_V(y, 2k\theta) 
	e^{-2{\rm i}\tau\theta\cdot y}{\rm d}y\Big]\\
	&= \frac{1}{16\pi^2} \Big[(2k)^{-m}\widehat{h}(2\tau\theta) + \int_{D} a_V(y, 2k\theta) 
	e^{-2{\rm i}\tau\theta\cdot y}{\rm d}y\Big],\end{align*} which implies that \begin{equation}\label{firststep}
\widehat{h}(2\tau\theta)=(16\pi^2)(2k)^{m} \mathbb E[u_0^\infty(-\theta,\theta, k) \overline{u_0^\infty(-\theta,\theta, k+\tau)}]+\mathcal{O}(1/k)
    \end{equation} uniformly holding for any $\tau>0$, $\theta \in \mathbb S^2$ and \textcolor{blue}{ $V\in\mathcal V$.} In \cite{Caro}, the authors have used an ergodicity theorem to show that almost surely \begin{equation*}
	\lim_{k \to \infty}\frac{1}{k}\int_k^{2k}s^{m}(\overline{u_0^\infty(-\theta,\theta, s+\tau)} u_0^\infty(-\theta,\theta, s)-\mathbb{E}[\overline{u_0^\infty(-\theta,\theta, s+\tau)} u_0^\infty(-\theta,\theta,s)])\,\mathrm{d}s=0.
\end{equation*} 
To further derive the stability, we may need to quantify the above limit. 
To this end, we prove an ergodicity result which improves the one utilized in \cite{Caro} and traces back to \cite[p. 94]{cramer1967stationary}.
\begin{Le}\label{Thm2.1}
	Let $X_t$ with $t \ge 0$ be a real-valued stochastic process with continuous paths. Furthermore, the process $X_t$ satisfies $\mathbb{E}X_t=0$ and there exist non-negative constants $c_1$ and $c_2$ satisfying $0 \le 2 c_1 < c_2 <1$ such that \[
	| \mathbb{E}(X_s X_t) |\lesssim \frac{s^{c_1}+t^{c_1}}{1+|s-t|^{c_2}}
	\] for any $s,t  \ge 0$. Then for any positive $\epsilon$ satisfying $\epsilon < (c_2-2c_1) /4$, we almost surely have \[
	\lim_{T \to \infty} \frac{1}{T^{1-\epsilon}} \int_0^T X_t \,\mathrm{d}t =0
	.\]
\end{Le}
\begin{proof}
	Denoting $Y_T=\frac{1}{T^{1-\epsilon}}\int_0^T X_t \,\mathrm{d}t$, we have  \begin{align*}
		\mathbb{E}Y_T^2 &=\frac{1}{T^{2-2\epsilon}} \int_0^T \int_0^T \mathbb{E}({X}_t {X}_s)\,\mathrm{d}t\,\mathrm{d}s \lesssim \frac{1}{T^{2-2\epsilon} }\int_0^T \int_0^T \frac{s^{c_1}+t^{c_1}}{1+|s-t|^{c_2}} \,\mathrm{d}t\,\mathrm{d}s \\ &\lesssim \frac{1}{T^{1-{c_1}-2\epsilon}}  \int_0^T \frac{1}{1+\tau^{c_2}}\,\mathrm{d}\tau \lesssim \frac{1}{T^{c_2-c_1-2\epsilon}}.
	\end{align*}
	Recalling ${c_2}-{c_1}-2\epsilon>0$, there exists a constant$\gamma>0$ such that $\gamma({c_2}-{c_1}-2\epsilon)>1$. Letting $T_n=n^\gamma$, we have \[
	\sum_{n=1}^\infty \mathbb{E}Y_{T_n}^2 \lesssim \sum_{n=1}^\infty \frac{1}{n^{\gamma({c_2}-{c_1}-2\epsilon)}} < \infty.
	\] Then applying the  Borel-Cantelli lemma and Chebyshev inequality, we derive \[
	\lim_{T_n \to \infty} \frac{1}{T_n^{1-\epsilon}} \int_0^{T_n} X_t \,\mathrm{d}t =0.
	\] Therefore, it suffices to show that almost surely we have \begin{equation}\label{2.3}
		\lim_{n\to \infty} \sup_{T_n \le T \le T_{n+1}} \left|\frac{1}{T^{1-\epsilon}} \int_0^T X_t \,\mathrm{d}t-\frac{1}{T_n^{1-\epsilon}} \int_0^{T_n} X_t \,\mathrm{d}t \right| =0.
	\end{equation}
	Denoting 
	$
	W_n=\sup_{T_n \le T \le T_{n+1}} \left|\frac{1}{T^{1-\epsilon}} \int_0^T X_t \,\mathrm{d}t-\frac{1}{T_n^{1-\epsilon}} \int_0^{T_n} X_t \,\mathrm{d}t \right|
	$, we have
	\begin{align*}
		W_n  &=  \sup_{T_n \le T \le T_{n+1}} \left| \left(\frac{1}{T^{1-\epsilon}}-\frac{1}{T_n^{1-\epsilon}} \right)\int_0^{T_n} X_t \,\mathrm{d}t+\frac{1}{T^{1-\epsilon}} \int_{T_n}^T X_t \,\mathrm{d}t \right| \\ &\le \frac{T_{n+1}^{1-\epsilon}-T_n^{1-\epsilon}}{T_n^{2-2\epsilon}} \int_0^{T_n} |X_t| \,\mathrm{d}t + \frac{1}{T_n^{1-\epsilon}} \int_{T_n}^{T_{n+1}} |T(t)|\,\mathrm{d}t,
	\end{align*} which gives \begin{align*}
		W_n^2 \le 2 \frac{(T_{n+1}^{1-\epsilon}-T_n^{1-\epsilon})^2}{T_n^{4-4\epsilon}} \int_0^{T_n}\int_0^{T_n} |X_t X_s|\,\mathrm{d}t\,\mathrm{d}s + \frac{2}{T_n^{2-2\epsilon}}\int_{T_n}^{T_{n+1}}\int_{T_n}^{T_{n+1}}|X_tX_s|\,\mathrm{d}t\,\mathrm{d}s.
	\end{align*}
	Using the estimate $
	\mathbb{E} |X_t X_s| \le \sqrt{\mathbb{E} |X_t|^2 \mathbb{E}|X_s|^2} \lesssim   (ts)^{c_1/2},
	$ we obtain \begin{align*}
		\mathbb{E} Y_n^2 \lesssim  T_n^{2+c_1}\frac{(T_{n+1}^{1-\epsilon}-T_n^{1-\epsilon})^2}{T_n^{4-4\epsilon}} + \frac{1 }{T_n^{2-2\epsilon}} (T_{n+1}^{1+c_1/2}-T_n^{1+c_1/2})^2.
	\end{align*} Since $T_n=n^\gamma$, direct calculation yields \[ \mathbb{E} W_n^2  \lesssim \frac{1}{n^{2-\gamma c_1-2\epsilon\gamma}}.\] AS $c_1+2 \epsilon < c_2-c_1-2\epsilon$, there exists $\gamma>0$ such that $\gamma(c_1+2 \epsilon)<1<\gamma(c_2-c_1-2\epsilon)$, which implies \[
	\sum_{n=1}^\infty \mathbb{E} W_n^2  \lesssim \sum_{n=1}^\infty \frac{1}{n^{2-\gamma c_1-2\epsilon\gamma}}< \infty.
	\] Finally, by applying the Borel-Cantelli lemma and the Chebyshev inequality again we obtain \eqref{2.3}. The proof is complete.
\end{proof} 

Next,
letting $
Z_k=u^\infty(-\theta,\theta, k)$, a direct calculation gives \begin{equation}\label{3.4}
	\overline{u_0^\infty(-\theta,\theta, k+\tau)} u_0^\infty(-\theta,\theta, k)-\mathbb{E}[\overline{u_0^\infty(-\theta,\theta, k+\tau) }u_0^\infty(\hat{x}, k)]= \sum_{U_k \in \mathcal{A}} C(U_k) (U_k^2-\mathbb E U_k^2),
\end{equation} where \begin{align*} \mathcal{A}:=\{\Re Z_k, \Im Z_k, \Re Z_{k+\tau}, \Im Z_{k+\tau},  \Re Z_k-\Re Z_{k+\tau}, \\ \Im Z_k-\Im Z_{k+\tau},\Re Z_{k+\tau}+ \Im Z_k, \Im Z_{k+\tau}-\Re Z_k\}\end{align*} and the constant $C(U_k) \in \{ \frac{1+i}{2},-\frac{1}{2},-\frac{i}{2}\}$.
Let $X_k:=k^{m} (U_k^2-\mathbb E U_k^2)$. According to \cite[Lemma 3.4]{li2022far}, for $V \in \mathcal V$, there exist a constant $C$ which does not depend on $\theta,\tau, k_1,k_2$ and $V$ such that \[
\mathbb{E}(X_{k_1} X_{k_2}) \le C \left((1+|k_1-k_2-\tau|)^{-1}+(1+|k_1-k_2+\tau|)^{-1}+(1+|k_1-k_2|)^{-1} \right).\]  Assuming that $\tau \in (0,1/2)$, 
the above estimate can be turned into \begin{equation}\label{unies}
\mathbb{E}(X_{k_1} X_{k_2}) \le C (1+|k_1-k_2|)^{-1} ,
\end{equation} where the constant $C$ does not depend on $\tau, k_1,k_2$ and $V$ either. Hence, applying Lemma \ref{Thm2.1} to $X_k$ gives 
\begin{equation*}
	\lim_{k \to \infty}\frac{1}{k^{1-\epsilon}}\int_k^{2k}s^{m}(\overline{u_0^\infty(-\theta,\theta, s+\tau)} u_0^\infty(-\theta,\theta, s)-\mathbb{E}[\overline{u_0^\infty(-\theta,\theta, s+\tau)} u_0^\infty(-\theta,\theta,s)])\,\mathrm{d}s=0
	\end{equation*} for any $\epsilon \in (0,1/4).$ It should be noticed that the above limit may not be uniformly valid for $\theta \in \mathbb S^2$, $\tau \in (0,1/2)$ and $V \in \mathcal V$, though the estimate \eqref{unies} uniformly holds. Thus, the ergodicity theory
	does not yield an almost surely stability result. To deal with this issue, by checking the proof of Lemma \ref{Thm2.1} we notice that it is possible to 
	derive a stability result in the probabilistic sense. In fact, denoting 
	\[Y_k:=\frac{1}{k^{1-\epsilon}}\int_k^{2k} X_s \,\mathrm{d}s,\]
	the proof of Lemma \ref{Thm2.1} implies that 
	\begin{align}\label{ps}
	\mathbb P(\omega: \exists \,  k \ge \tilde k, |Y_{ k}| \ge 1) \le \mathbb P(\omega: \cup_{m=n}^\infty\{|Y_{k_m}| \ge 1/2\})+\mathbb P(\omega: \cup_{m=n}^\infty\{|{W_m}| \ge 1/2\}),
\end{align} 
where $k_m=m^\gamma$, $W_m=\sup_{k_m \le k \le k_{m+1}} \left|\frac{1}{k^{1-\epsilon}} \int_k^{2k} X_t \,\mathrm{d}t-\frac{1}{k_m^{1-\epsilon}} \int_{k_m}^{2k_m} X_t \,\mathrm{d}t \right|$. Here $\tilde k$ is a positive constant and $n \in \mathbb N$ is chosen such that $\tilde k \in [n^\gamma, (n+1)^\gamma]$. Moreover, $\gamma$ is chosen such that $2\epsilon\gamma<1<\gamma(c_2-2\epsilon)$, where $c_2$ is a constant in $ (0,1)$ and $\epsilon <c_2/4$. Then we have \begin{align*}
    \mathbb P(\omega: \exists \,k \ge \tilde k, |Y_{ k}| \ge 1) \lesssim \frac{1}{n^{\gamma(c_2-2\epsilon)-1}}+\frac{1}{n^{1-2\epsilon\gamma}} \lesssim \frac{1}{{\tilde k}^{c_2-2\epsilon-1/\gamma}}+\frac{1}{{\tilde k}^{1/\gamma-2\epsilon}} 
\end{align*} by the proof of Lemma \ref{Thm2.1}. As a consequence, we obtain \begin{align}\label{pes}
    \mathbb P(\exists \, k \ge \tilde k, |Y_{k}| \ge 1) \lesssim  \frac{1}{\tilde k^{a(\epsilon)}},
\end{align} where $a(\epsilon):=\min\{c_2-2\epsilon-1/\gamma,1/\gamma-2\epsilon\} \in (0,1)$. It should also be noticed that the above estimate \eqref{pes} holds uniformly for any $\theta \in \mathbb S^2$, $\tau \in (0,1/2)$ and $V \in\mathcal V.$

On the other hand, when $|Y_{k}| < 1$ holds for all $k \ge \tilde k$ and $U_{ k} \in \mathcal A$, there holds the estimate \begin{align}\label{conditional}
    \left|\frac{1}{ k}\int_{k}^{2 k}s^{m}(\overline{u_0^\infty(-\theta,\theta, s+\tau)} u_0^\infty(-\theta,\theta, s)-\mathbb{E}[\overline{u_0^\infty(-\theta,\theta, s+\tau)} u_0^\infty(-\theta,\theta,s)])\,\mathrm{d}s\right| \lesssim \frac{1}{k^{\epsilon}}.
\end{align} Therefore, combining \eqref{firststep} and \eqref{conditional}, we have the following result. \begin{Le}\label{zeroorder} For any $\epsilon \in (0,1/4)$, if $|Y_{k}| < 1$ holds for all $k \ge \tilde k$, $U_{ k} \in \mathcal A$, then \begin{align}\label{4.12}
	\widehat{h}(2\tau\theta)=\frac{(16\pi^2)}{k}\int_k^{2k}(2s)^{m} [u_0^\infty(-\theta,\theta, s) \overline{u_0^\infty(-\theta,\theta, s+\tau)}]\,\mathrm{d}s+\mathcal{O}(1/k^\epsilon)
\end{align} uniformly holds for all $\tau \in (0,1/2)$, $\theta \in \mathbb S^2$ and $V \in \mathcal V$. \end{Le}

\subsection{Analysis of the remaining terms in Born series}
In view of Lemma \ref{zeroorder}, we shall estimate the difference between the correlation far-field pattern and the correlation of the zeroth order term in Born series. To achieve this goal, we investigate the first and second order terms of the form $\frac{1}{k}\int_k^{2k} s^{m} |u_j^\infty(-\theta,\theta,s)|^2\,\mathrm{d}s$ for $j=1, 2$.

We begin with the first order term.
The strategy is to not estimate directly $$\frac{1}{k}\int_k^{2k} s^{m} |u_1^\infty(-\theta,\theta,s)|^2\,\mathrm{d}s,$$ but to estimate the average term 
 $$\frac{1}{k}\int_{\mathbb S^2}\int_k^{2k} s^{m} |u_1^\infty(-\theta,\theta,s)|^2\,\mathrm{d}\theta\,\mathrm{d}s.$$ Before doing so, we give a useful lemma in \cite{wang2024stability}, which shows the kernel $K_V(x,y)$ can be decomposed as a sum of a singular part and a bounded continuous remainder. 
 \begin{Le}\label{Lem2.2}
For $V \in \mathcal V$, the covariance function $K_V(x,y)$ has the following form:
	
	(i)  If $a<m-2\le a+1$ with $a=1,2,3...$,  \[
	K_V(x, y)= ch(x)|x-y|^{m-3} +F_m(x,y),
	\] where $c$ is a constant dependent on $m$ and $F_m(x,y) \in C^{a,\gamma}(\mathbb{R}^3 \times \mathbb{R}^3)$ with $\gamma \in (0, m-2-a)$.
	
	(ii)
	If $2<m< 3$,  \[
	K_V(x, y)= ch(x)|x-y|^{m-3} +F_m(x,y),
	\] where $c$ is a constant dependent on $m$ and $F_m(x,y) \in C^{0,\gamma}(\mathbb{R}^3 \times \mathbb{R}^3)$ with $\gamma \in (0, m-2)$. 
	
	(iii)
	If $m=3$,  \[
	K_V(x, y)= ch(x)\log{|x-y|} +F_m(x,y),
	\] where $c$ is a constant dependent on $F_m(x,y) \in C^{0,\gamma}(\mathbb{R}^3 \times \mathbb{R}^3)$ with $\gamma \in (0, 1)$. 
	
	For all of the above three cases, we have \[
	\|F_m\|_{L^\infty(D \times D)} \lesssim M_1. 
	\]
\end{Le}
The following lemma provides an estimate for the first order reminder.
\begin{Le}\label{firstorder} Assume that $14/5<m<4$ and let $M_0>0$ be an arbitrarily chosen constant. There exists an event $A$ such that when it occurs, there exists a constant $\delta_0 \in (0,1)$ depending on $m$ such that
     \begin{equation}\label{5.3}
	\frac{1}{k^{1-\delta_0}}\int_{\mathbb S^2}\int_k^{2k}s^m|u^\infty_1(-\theta,\theta,s)|^2\,\mathrm{d}s\,\mathrm{d}\sigma(\theta) \le M_0
\end{equation}  and 
     \begin{equation}\label{5.3'}
	\frac{1}{k^{1-\delta_0}}\int_{\mathbb S^2}\int_k^{2k}s^m|u^\infty_1(-\theta,\theta,s+\tau)|^2\,\mathrm{d}s\,\mathrm{d}\sigma(\theta) \le M_0.
\end{equation} Furthermore, the probability of the event $A^c$ can be estimated by \begin{align}\label{evea}
    \mathbb P(A^c) \lesssim \frac{1}{M_0},
\end{align} which  uniformly holds for $\tau \in (0,1/2)$ and $V \in \mathcal V.$
 \end{Le}  
 \begin{proof}
Denote the event $A$ by \begin{equation}\label{5.4}A:=\left\{\omega:\int_{\mathbb S^2}\int_1^\infty s^{m-1+{\delta_0}}|u^\infty_1(-\theta,\theta,s+\tau)|^2\,\mathrm{d}s\,\mathrm{d}\sigma(\theta) \le \widetilde M_0\right\} \subset \Omega
\end{equation} for $\tau \in [0,1/2)$ and $V \in\mathcal V$, where $\widetilde M_0$ is any positive constant. If we choose $M_0=\widetilde M_0 2^{1-\delta_0}$,  both \eqref{5.3} and \eqref{5.3'} hold when $A $ occurs since now we have \begin{align*}
    &\frac{1}{k^{1-\delta_0}}\int_{\mathbb S^2}\int_k^{2k}s^m|u^\infty_1(-\theta,\theta,s+\tau)|^2\,\mathrm{d}s\,\mathrm{d}\sigma(\theta) \\ &\le 2^{1-\delta_0}\int_{\mathbb S^2}\int_1^\infty s^{m-1+{\delta_0}}|u^\infty_1(-\theta,\theta,s+\tau)|^2\,\mathrm{d}s\,\mathrm{d}\sigma(\theta) \le 2^{1-\delta_0}\widetilde M_0.
\end{align*} 

On the other hand, the probability of $A^c$ can by estimated by\begin{align*}
    \mathbb P(A^c) \le \frac{\mathbb E\Big[\int_{\mathbb{S}^2}\int_1^\infty s^{m-1+\delta_0}|u^\infty_1(-\theta,\theta,s+\tau)|^2\,\mathrm{d}s \,\mathrm{d}\sigma(\theta)\Big]}{\widetilde M_0}.
\end{align*}
Therefore, if the estimate \begin{equation}\label{5.5}
	\mathbb E \int_{\mathbb{S}^2}\int_1^\infty s^{m-1+\delta_0}|u^\infty_1(-\theta,\theta,s+\tau)|^2\,\mathrm{d}s \,\mathrm{d}\sigma(\theta)< \infty
\end{equation} uniformly holds for $\tau \in [0,1/2)$ and $V \in\mathcal V$, then  \[
\mathbb P(A^c) \lesssim \frac{1}{M_0}.
\] 

In what follows, we prove that \eqref{5.5} uniformly holds for $\tau \in [0,1/2)$ and $V \in\mathcal V$. For convenience, we consider the case $\tau=0$. For the other cases $\tau \in (0,1/2)$, \eqref{5.5} can be proven in a similar manner.
To achieve this goal, we consider the mollification $V_\epsilon := V * \chi_\epsilon$, where $\chi_\epsilon(x)=\frac{1}{\epsilon^3}\chi(\frac{x}{\epsilon})$ with $\chi \in C_0^\infty(\mathbb{R}^3)$ such that $\int_{\mathbb{R}^3} \chi \,\mathrm{d}x=1$. By the Fat$\acute{\rm o}$u lemma, \eqref{5.5} holds if \begin{equation}\label{eq5.6}
	\limsup_{\epsilon \to 0} \mathbb E \int_{\mathbb{S}^2}\int_1^\infty s^{m-1+\delta_0}|u^\infty_{1, \epsilon}(-\theta,\theta,s)|^2\,\mathrm{d}s\,\mathrm{d}\sigma(\theta) < \infty,
\end{equation}  where \[
u^\infty_{1, \epsilon}(-\theta,\theta,s)=\int_{\mathbb{R}^3}\int_{\mathbb{R}^3} e^{ik\theta\cdot(x+y)} V_\epsilon(x) V_\epsilon(y)\frac{e^{ik|x-y|}}{4\pi|x-y|}\,\mathrm{d}x\,\mathrm{d}y.
\] Taking $z=x+y$, $z'=x-y$ and then letting $z'=\rho \omega$, we obtain \begin{align*}
	u^\infty_{1, \epsilon}(-\theta,\theta,s) &= \int_{\mathbb{R}^3}\int_{\mathbb{R}^3}e^{ik\theta \cdot z}V_\epsilon\left(\frac{z+z'}{2}\right)V_\epsilon\left(\frac{z-z'}{2}\right)\frac{e^
		{ik|z'|}}{4\pi|z'|}\,\mathrm{d}z\,\mathrm{d}z'\\ &=\int_{\mathbb{R}}\int_{\mathbb{S}^2}\int_{\mathbb{R}^3}\frac{e^
		{ik(\rho+\theta \cdot z)}}{4\pi}V_\epsilon\left(\frac{z+\rho \omega}{2}\right)V_\epsilon\left(\frac{z-\rho\omega}{2}\right)\,\mathrm{d}z\,\mathrm{d}\sigma(\omega)\,\mathrm{d}\rho.
\end{align*} Denote \begin{align*}
	I_\epsilon(z,\rho)=\rho 1_{[0,\infty)}\int_{\mathbb{S}^2}V_\epsilon\left(\frac{z+\rho \omega}{2}\right)V_\epsilon\left(\frac{z-\rho\omega}{2}\right)\,\mathrm{d}\sigma(\omega).
\end{align*} From \cite[Lemma 4.7]{Caro} we have \begin{align*}
	&\int_{\mathbb{S}^2}\int_1^\infty s^{m-1+\delta_0} |u^\infty_{1, \epsilon}(-\theta,\theta,s)|^2\,\mathrm{d}s\,\mathrm{d}\sigma(\theta) \\ &\lesssim \int_{\mathbb R^3}\int_1^\infty s^{m-3+\delta_0} \left| \int_{\mathbb R} e^{is\rho} I_\epsilon(z,\rho)\,\mathrm{d}\rho\right|^2\,\mathrm{d}s\,\mathrm{d}z \\ &\lesssim \int_{\mathbb R^3}\int_1^\infty  s^{ N_m}\left| \int_{\mathbb R} e^{is\rho} I_\epsilon(z,\rho)\,\mathrm{d}\rho\right|^2\,\mathrm{d}s\,\mathrm{d}z,
\end{align*} where \begin{align*}
	N_m= \begin{cases}
		1, &\quad 3 \le m  < 4, \\   0, &\quad \frac{14}{5} < m  < 3,
	\end{cases}\quad \delta_0=\begin{cases}
		4-m, &\quad 3 \le m  < 4, \\   3-m, &\quad \frac{14}{5} < m  < 3.
	\end{cases}
\end{align*} Hence we have \begin{align*}
	&\int_{\mathbb{S}^2}\int_1^\infty s^{m-1+\delta_0} |u^\infty_{1, \epsilon}(-\theta,\theta,s)|^2\,\mathrm{d}s\,\mathrm{d}\sigma(\theta) \\ &\lesssim\int_{\mathbb R^3}\int_{\mathbb R}(-i)^{N_m}\partial_\rho^{N_m} I_\epsilon(z,\rho)\overline{I_\epsilon(z,\rho)}\,\mathrm{d}\rho\,\mathrm{d}z.
\end{align*} Since $I_\epsilon$ has compact support,  there exists $R_0>0$ such that \begin{align*}
	&\limsup_{\epsilon \to 0} \mathbb E \int_{\mathbb{S}^2}\int_1^\infty s^{m-1+\delta_0}|u^\infty_{1, \epsilon}(-\theta,\theta,s)|^2\,\mathrm{d}s\,\mathrm{d}\sigma(\theta) \\ &\lesssim \limsup_{\epsilon \to 0} \mathbb E \int_{|z|<R_0} \int_{|\rho|<R_0} (-i)^{N_m}\partial_\rho^{N_m} I_\epsilon(z,\rho)\overline{I_\epsilon(z,\rho)} \,\mathrm{d}\rho\,\mathrm{d}z. 
\end{align*}

We first consider the case $3 \le m <4$ which gives $N_m=1$.
Notice that there holds \begin{align*}
	\partial_\rho I_\epsilon(z,\rho) &=1_{[0,\infty)}(\rho)\int_{\mathbb{S}^2}V_\epsilon\left(\frac{z+\rho \omega}{2}\right)V_\epsilon\left(\frac{z-\rho\omega}{2}\right)\,\mathrm{d}\sigma(\omega)\\ &\quad+\rho 1_{[0,\infty)}(\rho)\int_{\mathbb{S}^2}\nabla V_\epsilon\left(\frac{z+\rho \omega}{2}\right)\cdot \frac{\omega}{2}V_\epsilon\left(\frac{z-\rho\omega}{2}\right)\,\mathrm{d}\sigma(\omega)\\ &\quad-\rho 1_{[0,\infty)}(\rho)\int_{\mathbb{S}^2}\nabla V_\epsilon\left(\frac{z-\rho \omega}{2}\right)\cdot \frac{\omega}{2}V_\epsilon\left(\frac{z+\rho\omega}{2}\right)\,\mathrm{d}\sigma(\omega)\\ &:= I_{1,\epsilon}(z,\rho)+I_{2,\epsilon}(z,\rho)+I_{3,\epsilon}(z,\rho).
\end{align*} Then we have \[\mathbb E [ \partial_\rho I_\epsilon(z,\rho)\overline{I_\epsilon(z,\rho)}]=\sum_{j=1}^3\mathbb{E} [I_{j,\epsilon}(z,\rho)\overline{I_{\epsilon}(z,\rho)}]=\sum_{j=1}^3\mathbb{E} [I_{j,\epsilon}(z,\rho){I_{\epsilon}(z,\rho)}].\]
In what follows, we only estimate the term $\mathbb{E} [I_{2,\epsilon}(z,\rho)I_{\epsilon}(z,\rho)]$ since the others can be treated in a similar way. The term $\mathbb{E} [I_{2,\epsilon}(z,\rho)I_{\epsilon}(z,\rho)]$ can be expressed by \begin{align*}
	&\mathbb{E} [I_{2,\epsilon}(z,\rho)I_{\epsilon}(z,\rho)] \\
	&=\rho^21_{[0,\infty)}(\rho)  \int_{\mathbb{S}^2}\int_{\mathbb{S}^2}\mathbb{E}[X_1X_2X_3X_4]\,\mathrm{d}\sigma(\omega)\,\mathrm{d}\sigma(\tilde{\omega}) \\ &=\rho^21_{[0,\infty)}(\rho)\Big\{\int_{\mathbb{S}^2}\int_{\mathbb{S}^2}\mathbb{E}[X_1X_2] \mathbb E [X_3X_4]\,\mathrm{d}\sigma(\omega)\,\mathrm{d}\sigma(\tilde{\omega}) \\&\quad+ \int_{\mathbb{S}^2}\int_{\mathbb{S}^2}\mathbb{E}[X_1X_3]\mathbb E [ X_2X_4]\,\mathrm{d}\sigma(\omega)\,\mathrm{d}\sigma(\tilde{\omega})\\&\quad+ \int_{\mathbb{S}^2}\int_{\mathbb{S}^2}\mathbb{E}[X_1X_4]\mathbb E [ X_2X_3]\,\mathrm{d}\sigma(\omega)\,\mathrm{d}\sigma(\tilde{\omega})\Big\},
\end{align*} where we have applied Isserli's theorem \cite{Caro} since 
\begin{align*}
	X_1:=\nabla V_\epsilon\left(\frac{z+\rho \omega}{2}\right)\cdot \frac{\omega}{2},& \quad X_2:=V_\epsilon\left(\frac{z-\rho\omega}{2}\right),\\  X_3:= V_\epsilon\left(\frac{z+\rho \tilde{\omega}}{2}\right),& \quad X_4:=V_\epsilon\left(\frac{z-\rho\tilde{\omega}}{2}\right),
\end{align*} are all Gaussian random variables. We also have \begin{align*}
	\lim_{\epsilon \to 0}\mathbb E [X_1X_2] &= \lim_{\epsilon \to 0} \mathbb E \left[\nabla V_\epsilon\left(\frac{z+\rho \omega}{2}\right)\cdot \frac{\omega}{2}V_\epsilon\left(\frac{z-\rho\omega}{2}\right)\right]\\ &=\nabla_1 K_V\left(\frac{z+\rho \omega}{2},\frac{z-\rho \omega}{2}\right) \cdot \frac{\omega}{2},
\end{align*}  where $\nabla_1 K_V(x,y):=\nabla_x K_V(x,y)$. Here $K_V(x, y)$ is the kernel of the covariance operator of $V$.
From Lemma \ref{Lem2.2} we know that for $m >2$, $\nabla_x K_V(x,y)$ is weakly singular. Thus, the above limit holds pointwisely. Similar conclusions hold for the other $\mathbb E[X_i X_j] $. From Lemma \ref{Lem2.2} we have that for $m \in [3,4)$ the estimate \begin{align*}
	&\limsup_{\epsilon \to 0} \mathbb E\int_{|z|<R_0} \int_{|\rho|<R_0}[I_{2,\epsilon}(z,\rho)I_{\epsilon}(z,\rho)] \,\mathrm{d}\rho\,\mathrm{d}z\\& \lesssim \int_{|z|<R_0} \int_0^{R_0} \rho^2\int_{\mathbb{S}^2}\int_{\mathbb{S}^2} |\rho|^{-1}|\log{(\rho \tilde{\omega})}||\tilde{\omega}|+|\rho(\omega-\tilde{\omega})|^{-1}|\omega||\log{(\rho(\tilde{\omega}-\tilde{\omega}))}| \\ &\quad+|\rho(\omega+\tilde{\omega})|^{-1}|\omega||\log{(\rho(\tilde{\omega}+\omega))}|\,\mathrm{d}\sigma(\tilde{\omega})\,\mathrm{d}\sigma(\omega)\,\mathrm{d}\rho\,\mathrm{d}z < \infty
\end{align*} uniformly holds for $V \in \mathcal V$, which completes the proof of \eqref{eq5.6} for $m \in [3,4)$. Using similar arguments we can also prove \eqref{eq5.6} for
 $m \in (\frac{14}{5},3)$.
which gives \eqref{5.5}. In summary, we have obtained  \[
\mathbb E \int_{\mathbb{S}^2}\int_1^\infty s^{m-1+\delta_0}|u^\infty_1(-\theta,\theta,s)|^2\,\mathrm{d}s \,\mathrm{d}\sigma(\theta)< C,
\] where \[\delta_0=\begin{cases}
	4-m, &\quad 3 \le m  < 4, \\   3-m, &\quad \frac{14}{5} < m  < 3.\end{cases}\] Here $C $ is a constant which does not depend on $V \in \mathcal V$.
	The proof is complete.
    \end{proof}
    
The estimates \eqref{5.3}--\eqref{5.3'} yield the following estimates 
\begin{align}\label{5.7}
	\frac{1}{k}\int_k^{2k}\int_{\mathbb S^2}s^m|u^\infty_1(-\theta,\theta,s)|^2\,\mathrm{d}\sigma(\theta)\,\mathrm{d}s \le \frac{M_0}{k^{\delta_0}}
\end{align} and \begin{align}\label{5.7'}
	\frac{1}{k}\int_k^{2k}\int_{\mathbb S^2}s^m|u^\infty_1(-\theta,\theta,s+\tau)|^2\,\mathrm{d}\sigma(\theta)\,\mathrm{d}s \le \frac{M_0}{k^{\delta_0}}.
\end{align} 

Next, we turn to the second-order term. The following lemma provides an estimate of the the second-order term in the Born series.
\begin{Le}\label{highorder} There exists a positive constant $C_1$ such that if $K_0>C_1\|V\|_{W^{\alpha,p}(\mathbb R^3)}$, there holds the following estimates for $k>K_0$:
\begin{align}\label{5.7.7}
	\frac{1}{k}\int_k^{2k}\int_{\mathbb S^2}s^m|u^\infty_2(-\theta,\theta,s)|^2\,\mathrm{d}\sigma(\theta)\,\mathrm{d}s \lesssim \frac{K_0^{3+4\alpha}}{k^{4+12\alpha-m}},
\end{align} and \begin{align}\label{5.7.7,}
	\frac{1}{k}\int_k^{2k}\int_{\mathbb S^2}s^m|u^\infty_2(-\theta,\theta,s+\tau)|^2\,\mathrm{d}\sigma(\theta)\,\mathrm{d}s \lesssim \frac{K_0^{3+4\alpha}}{k^{4+12\alpha-m}}.
\end{align} Here \eqref{5.7.7,} uniformly holds for $\tau \in (0,1/2)$.
\end{Le}
\begin{proof} By Lemma \ref{Lem4.1}--\ref{Lem5.1}, there exists a positive constant $C_2$ such that $\|\chi R_0(k)V\|_{\mathcal L(H^{-\alpha},H^{-\alpha})} \le C_2k^{-1}\|V\|_{W^{\alpha,p}(\mathbb R^3)}$. Hence we choose $C_1 \ge 2C_2$ such that $\|\chi R_0(k)V\|_{\mathcal L(H^{-\alpha},H^{-\alpha})} \le 1/2$ if $k>C_1\|V\|_{W^{\alpha,p}(\mathbb R^3)}$. Furthermore, we have 
\begin{align}
	|u_2^\infty(-\theta,\theta,k)| &\le \sum_{j \ge 2} \Big| \int_{\mathbb{R}^3}e^{ik \theta \cdot  y}V(y) [( R_0(k)V)^jf](y)\,\mathrm{d}y\Big| \notag \\ &\lesssim  \sum_{j \ge 2} \|e^{ik\theta \cdot y}V\|_{H^{\alpha}} \| \chi (R_0(k) V)^j u^{inc}\|_{H^{-\alpha}} \notag \\ &\lesssim K_0\sum_{j \ge 2} \|e^{ik\theta \cdot y}\|_{H^{-\alpha}} (K_0/k)^{j(1+2\alpha)}\|u^{inc}\|_{H^{-\alpha}} \notag \\ & \lesssim K_0k^{-2\alpha}\sum_{j \ge 2} (K_0/k)^{j(1+2\alpha)} \lesssim K_0^{3+4\alpha}k^{-2-6\alpha},\label{5.8}
\end{align} which yields \eqref{5.7.7}. The estimate \eqref{5.7.7,} can be obtained in a similar way.
\end{proof} 

\subsection{Proof of the main theorem}
Based on the above analysis for the zeroth, first and second oder terms in the Born series expansion, 
we are in the position to prove Theorem \ref{Thm5.1}.

We begin by showing that the Fourier transform of the micro-correlation strength can be estimated by the far-field data and a high frequency tail
at a single realization in some event. 
For $14/5<m<4$, recalling $\alpha<(m-3)/2$, we have $4+12\alpha-m \in (0,1)$.
Denote the event
\begin{equation}\label{event}
	E:=A^{(1)} \cap A^{(2)}\cap\{\omega:|Y_k^{(j)}| < 1 \,\text{holds for all }k>\tilde k, \, W_k^{(j)} \in \mathcal A, K_0 > C_1\|V_j\|_{W^{\alpha,p}(\mathbb R^3)}, j=1,2\}\end{equation}
with $\epsilon \in (0,1/4)$.
Under the event $E$ and conditions of Theorem \ref{Thm5.1}, in what follows
	we derive a stability estimate at a single realization.
	We begin by taking subtraction of the correlation of the far-field pattern $u^\infty_{(1)}$ and $u^\infty_{(2)}$ for $k \ge \max \{\tilde k, K_0\}$, where $\tilde k$ is the constant in the probabilistic estimate \eqref{ps} to be determined later.
By doing so, using Lemma \ref{zeroorder} we obtain
\begin{align*}&\widehat{(h_1-h_2)}(2\tau\theta) \\ &=\frac{(16\pi^2)}{k}\int_k^{2k}(2s)^{m} [u_{0,(1)}^\infty(-\theta,\theta, s) \overline{u_{0,(1)}^\infty(-\theta,\theta, s+\tau)}-u_{0,(2)}^\infty(-\theta,\theta, s) \overline{u_{0,(2)}^\infty(-\theta,\theta, s+\tau)}]\,\mathrm{d}s \\ 
&\quad+\mathcal{O}(1/k^\epsilon),\end{align*} which implies   \begin{align}
&\int_{\mathbb S^2}|\widehat{(h_1-h_2)}(2\tau\theta)|^2\,\mathrm{d}\theta \notag\\ & \lesssim \frac{1}{k}\int_{\mathbb S^2}\int_k^{2k}s^{2m} [u_{0,(1)}^\infty(-\theta,\theta, s) \overline{u_{0,(1)}^\infty(-\theta,\theta, s+\tau)}-u_{0,(2)}^\infty(-\theta,\theta, s) \overline{u_{0,(2)}^\infty(-\theta,\theta, s+\tau)}]^2\,\mathrm{d}s\,\mathrm{d}\theta \notag\\ 
&\quad+\frac{1}{k^{2\epsilon}} \label{sub1}
\end{align} by using Cauchy--Schwarz inequality. Combining Lemma \ref{firstorder}--\ref{highorder} gives \begin{align}\label{sub2}
    \frac{1}{k}\int_{\mathbb S^2}\int_k^{2k}s^m|u_{(j)}^\infty(-\theta,\theta,s)-u_{0,(j)}^\infty(-\theta,\theta,s)|^2\,\mathrm{d}s\,\mathrm{d}\sigma(\theta) \lesssim  \frac{M_0+K_0^{3+4\alpha}}{k^{\gamma(m)}}, \quad j=1,2,
\end{align} where $\gamma(m)$ is a positive constant only depends on $m.$ It should be noticed that both \eqref{sub1} and \eqref{sub2} uniformly hold for $\tau \in (0,1/2)$ and $V_1,V_2 \in \mathcal V.$ In view of \eqref{sub1} and \eqref{sub2} we have for $r\leq 1$
	\begin{align} \label{5.9}
		\int_{\mathbb S^2}|\widehat{(h_1-h_2)}(r\theta)|^2\,\mathrm{d}\sigma(\theta)\lesssim \sup_{0<\tau<1/2}\frac{1}{k}\int_k^{2k}\int_{\mathbb S^2}|s^mU(s,\theta,\tau)|^2\,\mathrm{d}\sigma(\theta)\,\mathrm{d}s
		+ \frac{M_0+K_0^{3+4\alpha}}{k^{\gamma(m,\epsilon)}},
	\end{align} where
	\[U(s,\theta,\tau)=\overline{u_{(1)}^\infty(-\theta, \theta,s+\tau)} {u_{(1)}^\infty(-\theta;\theta,  s)}-\overline{u_{(2)}^\infty(-\theta, \theta,s+\tau)} {u_{(2)}^\infty(-\theta;\theta,  s)}.\] 
	Here $\gamma(m,\epsilon)$ is a positive constant which only depends on $m$ and $\epsilon.$
	
	Recall that the data is given by
	\[
	\varepsilon^2 =  \sup_{k\in I, \tau\in(0, 1/2), } \varepsilon^2(k, \tau),
	\]
	where
	\[
	\varepsilon^2(k, \tau)=\frac{1}{k}\int_k^{2k}\int_{\mathbb S^2}|s^{m}U(s,\theta,\tau)|^2\,\mathrm{d}\sigma(\theta)\,\mathrm{d}s.
	\] Since the potential function is real-valued, we have $\overline{u(x, k)} = u(x,-k)$ and then $\overline{u^\infty(\hat x, \theta, k)} = u^\infty(\hat x, \theta, -k)$.
	Thus, we can meromorphically extend $\varepsilon^2(\cdot,\tau)$ from $\mathbb R^+$ to $\mathbb C $ by \begin{align*}
		\varepsilon^2(k, \tau)  =\frac{1}{k}\int_k^{2k} s^{2m}\int_{\mathbb S^2}U_1(s,\theta,\tau)U_2(s,\theta,\tau)\,\mathrm{d}\sigma(\theta)\,\mathrm{d}s,
	\end{align*} where \[
	U_1(s,\theta,\tau)=u_{(1)}^\infty(-\theta, \theta,-s-\tau) {u_{(1)}^\infty(-\theta,\theta,  s)}-u_{(2)}^\infty(-\theta, \theta,-s-\tau) {u_{(2)}^\infty(-\theta,\theta,  s)},
	\]\[U_2(s,\theta,\tau)=u_{(1)}^\infty(-\theta, \theta,s+\tau) {u_{(1)}^\infty(-\theta,\theta,  -s)}-u_{(2)}^\infty(-\theta, \theta,s+\tau) {u_{(2)}^\infty(-\theta,\theta,  -s)}.\] 
    Denote an infinite slab in $\mathbb C$ as follows
\begin{equation}\label{regionR}
	\mathcal{R} = \{z\in \mathbb C: (K_0, +\infty)\times (-h_0, h_0) \},
\end{equation}
for some constants $K_0>0$ and $h_0>0$.
The following lemma proved by \cite{LZZ} provides an analytic continuation principle in the infinite slab.
\begin{Le}\label{Lem4.6}
	Let $p(z)$ be analytic in the infinite rectangular slab $\mathcal{R}$
	and continuous in $\overline{\mathcal{R}}$ satisfying
	\begin{align*}
		\begin{cases}
			|p(z)|\leq \epsilon, &\quad z\in (K_0, K],\\
			|p(z)|\leq M, &\quad z\in \mathcal{R},
		\end{cases}
	\end{align*}
	where $K_0, K, \epsilon$ and $M$ are positive constants. Then there exists a function $\mu(z)$ with $z\in (K, +\infty)$ satisfying 
	\begin{equation*}
		\mu(z) \geq \frac{64ah_0}{3\pi^2(a^2 + 4h_0^2)} e^{\frac{\pi}{2h_0}(\frac{a}{2} - z)},
	\end{equation*}
	where $a = K - K_0$, such that
	\begin{align*}
		|p(z)|\leq M\epsilon^{\mu(z)}\quad \forall\, z\in (K, +\infty).
	\end{align*}
\end{Le}
In order to apply Lemma \ref{Lem4.6}, we next show that $\varepsilon^2(k, \tau)$ is analytic and has an upper bound for $k \in \mathcal{R}$.
\begin{Le}\label{Lem4.7}
	When $K_0>C_0\|V_j\|^2_{W^{\alpha,p}(\mathbb R^3)}, j=1, 2$, there exists an infinite slab $\mathcal{R}$ defined by \eqref{regionR} such that $\varepsilon^2(k,\tau)$ is analytic for $k \in \mathcal{R}$ and there holds the inequality \[
	|\varepsilon^2(k, \tau)| \lesssim K_0^2|k|^{2m-4\alpha},\quad k \in \mathcal{R}, \quad\tau \in (0,1/2).
	\] 
\end{Le}
\begin{proof}
	Letting $s=kt$ with $t \in (1,2)$, we obtain\begin{align*}
		\varepsilon^2(k, \tau)  =k^{2m}\int_1^{2} t^{2m}\int_{\mathbb S^2}\widetilde U_1(t,\theta,\tau) \widetilde U_2(t,\theta,\tau)\,\mathrm{d}\sigma(\theta)\,\mathrm{d}t,
	\end{align*} where \[
	\widetilde U_1(t)=u_{(1)}^\infty(-\theta, \theta,-kt+\tau ) {u_{(1)}^\infty(-\theta,\theta,  kt)}-u_{(2)}^\infty(-\theta, \theta,-kt+\tau ) {u_{(2)}^\infty(-\theta,\theta,  kt)},
	\]\[\widetilde U_2(t)=u_{(1)}^\infty(-\theta, \theta,kt+\tau )) {u_{(1)}^\infty(-\theta,\theta,  -kt)}-u_{(2)}^\infty(-\theta, \theta,kt+\tau ) {u_{(2)}^\infty(-\theta,\theta,  -kt)}.\] From Lemma \ref{Lem4.6}, when $K_0>C_0\|V\|^2_{W^{\alpha,p}(\mathbb R^3)}$, there exists $h_0$ such that  that $ \pm kt,  \pm (kt+\tau) \in \mathscr{S}$ for $k \in \mathcal{R},$ $t \in (1,2)$ and $\tau \in (0,1/2)$. Hence, from Lemma \ref{Lem5.2} we have that $\varepsilon^2(k,\eta)$ is analytic for $k \in \mathcal R.$ Furthermore, letting $\chi \in C_0^\infty(\mathbb{R}^3)$ be a cut-off function with $\chi|_{D}=1$, 
	we have the estimate\begin{align}
		|u_{(j)}^\infty(\widehat x,\theta,kt)| &= \Big| \frac{1}{4\pi} \int_{\mathbb{R}^3}e^{-ik\hat{x}\cdot y} V_j(y)(e^{ik\theta \cdot y}+R_{V_j}(k)(u^{inc}V_j)(y))\,\mathrm{d}y \Big| \notag\\ &\lesssim \|e^{-iky\cdot (\hat x-\theta)} V_j\|_{H^{\alpha}}\|\chi\|_{H^{-\alpha}}+\|R_{V_j}(k)(u^{inc}V_j)\|_{H^{-\alpha}}\|e^{-ik\hat{x}\cdot y} V_j\|_{H^{\alpha}} \notag\\ & \lesssim \sqrt{K_0}(|k|^{-\alpha}+|k|^{-1-4\alpha}) \lesssim \sqrt{K_0}|k|^{-\alpha}. \label{4.14}\end{align}  
	Similarly, for $k \in \mathcal{R}$, there hold the inequalities \begin{align}\label{4.15}
		|u^\infty(\hat{x},\theta,-kt)| \lesssim  \sqrt{K_0}|k|^{-\alpha}
	\end{align} and \begin{align}\label{4.16}
		|u^\infty(\hat{x},\theta,\pm (kt+\tau))| \lesssim \sqrt{K_0}|k|^{-\alpha}.
	\end{align} It also should be noticed that \eqref{4.16} uniformly holds for $\tau \in (0,1/2)$. Combining \eqref{4.14}--\eqref{4.16} we complete the proof. 
\end{proof}

Now we are in the position to derive the stability estimate.
We begin by deriving a stability for a single realization by
applying an argument of analytic continuation
developed in \cite{ZZ} under the event $\widetilde E:= E \cap \{\omega \in \Omega:K_0>C_0\|V_j\|^2_{W^{\alpha,p}(\mathbb R^3)} ,j=1,2\}$, where $E$ is given in \eqref{event}. Then for the complementary event $\widetilde E^c$ when $\widetilde E$ does not happen, we obtain a probabilistic estimate using \eqref{ps}. The stability follows by combing the two parts.

 Under the condition $K_0>C_0\|V_j\|^2_{W^{\alpha,p}(\mathbb R^3)} (j=1,2)$, applying the quantitative analytic continuation principle in Lemma \ref{Lem4.6} to $\varepsilon^2(k, \tau )$ yields
	\[
	|\varepsilon^2(k, \tau )|\lesssim K_0^2 k^{2m-4\alpha}\varepsilon^{2\mu({k})},\quad {k}>K,
	\]
	where 
	\[
	\mu({k})=\frac{64ah_0}{3\pi^2(a^2+4h_0^2)}e^{\frac{\pi}{2 h_0}(\frac{a}{2}-{k})}, \quad a = K - K_0.
	\]
	Taking $k=A>K$ we have
	\[
	\mu({A})\gtrsim ce^{-\sigma A}
	\]
	for some constant $\sigma>0$ and then
	\[
	|\varepsilon^2(A,\tau)|\lesssim K_0^2 A^{2m-4\alpha} \exp\{-ce^{-\sigma A}|\ln \varepsilon|\}.
	\]
	Using the fact 
	$e^{-x}\leq\frac{6!}{x^{6}}$ for $x>0$, we obtain the following estimate
	\[
	|\varepsilon^2(A,\tau)|\lesssim K_0^2 
    A^{2m-4\alpha} e^{6\sigma A}|\ln \varepsilon|^{-6},
	\]
	which holds uniformly for $\tau\in (0, 1/2)$. 
		
	Next, we consider the following two situations.
    
	\textbf{Case 1:} $K\leq \frac{1}{2\sigma}\ln|\ln \varepsilon|$.
	Taking $A=\frac{1}{2\sigma}\ln|\ln \varepsilon|$, we have, in particular,
	\[
	|\varepsilon^2(A,  \eta )|\lesssim A^{2+2m-4\alpha}|\ln\varepsilon|^{-3}.
	\]
	Since the event $E$ does not happen, by \eqref{5.9} we have
	\[
	\int_{\mathbb S^2}| \hat{h}(  r \theta)|^2\,\mathrm{d}\sigma(\theta)\lesssim A^{2+2m-4\alpha}|\ln\varepsilon|^{-3}+\frac{M_0+K_0^{3+4\alpha}}{A^{\gamma(\epsilon,m)}}\lesssim \frac{M_0+K_0^{3+4\alpha}}{(\ln|\ln\varepsilon|)^{2\beta_1}} \le \frac{M_0+K_0^{3+4\alpha}}{(K\ln|\ln\varepsilon|)^{\beta_1}}
	\]
	for $| r|\le 1$.
	Here we denote $\beta_1=\frac{\gamma(\epsilon,m)}{2}  - \frac{3t}{2}$.
    
	\textbf{Case 2:} $K\geq \frac{1}{2\sigma}\ln|\ln \varepsilon|$. Since \[
    |\varepsilon^2(K,  \eta )|\le \varepsilon^2,
    \] we have
	\[
	\int_{\mathbb S^2}| \widehat{h}( r \theta)|^2\mathrm{d} \theta\le\varepsilon^2+(K_0^{3+4\alpha}+M_0)K^{-2\beta_1} \le \varepsilon^2+\frac{K_0^{3+4\alpha}+M_0}{(K\ln|\ln\varepsilon|)^{\beta_1}},
	\] which holds for $|r| \le 1.$
	Combining Case 1 and Case 2 we obtain the following estimate
	\begin{align}\label{inte}
		\|\hat h\|_{L^2(B_1)}^2
		\lesssim \varepsilon^2+\frac{K_0^{3+4\alpha}+M_0}{K^\beta(\ln|\ln \varepsilon|)^\beta}.
	\end{align}
    At last, applying the analytic continuation as the proof of \cite[Theorem 2.1]{Li2016increasing}, for $h \in \mathcal C_r$, there holds the estimate \begin{equation} \label{oute}
		\|h\|^2_{L^2(D)}=\|\hat{h}\|^2_{L^2(\mathbb{R}^3)} \lesssim \frac{1}{\left|\ln{\|\hat{h}\|_{L^2(B_1)}^2} \right|^{\beta_2}}
	\end{equation}
	with $\beta_2=(2r-3)/4$, which yields the estimate in $\mathbb P(\cdot)$.
	
  Now we consider the complementary event for $\widetilde E$,
	which is denoted by 
	\[
	\widetilde{E}^c:=A^{(1),c}\cup E_1 \cup E_2,\] where
    \begin{align*} E_j:=\{\omega:\exists \,k \ge K, |Y_k^{(j)}| \ge 1 \,\text{holds for some } W_k^{(j)} \in \mathcal A\} &\cup\{\omega: K_0 \le C_1\|V_j\|_{W^{\alpha,p}(\mathbb R^3)} \} \\ &\cup\{\omega: K_0 \le C_0 \|V_j\|^2_{W^{\alpha,p}(\mathbb R^3)}\}
	\end{align*} for $j=1,2.$
	Here we let $\tilde k=K$.
	Recalling \eqref{pes}, the probability of the event $\widetilde{E}^c$ can be estimated by \begin{align}
	    \mathbb P(\widetilde{E}^c) &\lesssim \frac{2}{K^{a(\epsilon)}}+\frac{2}{M_0}+\mathbb P(K_0 \le C_1\|V_1\|_{W^{\alpha,p}(\mathbb R^3))})+\mathbb P(K_0  \le C_1\|V_2\|_{W^{\alpha,p}(\mathbb R^3)}) \\ &\quad+\mathbb P(K_0 \le C_0\|V_1\|^2_{W^{\alpha,p}(\mathbb R^3))})+\mathbb P(K_0  \le C_0\|V_2\|^2_{W^{\alpha,p}(\mathbb R^3)})\notag\\ & \lesssim\frac{1}{K^{a(\epsilon)}}+\frac{1}{M_0}+\frac{\mathbb E\|V_1\|_{W^{\alpha,p}(\mathbb R^3)}+\mathbb E\|V_2\|_{W^{\alpha,p}(\mathbb R^3)}}{K_0} +\frac{\mathbb E\|V_1\|_{W^{\alpha,p}(\mathbb R^3)}+\mathbb E\|V_2\|_{W^{\alpha,p}(\mathbb R^3)}}{\sqrt K_0}\notag\\ &\lesssim\frac{1}{K^{a(\epsilon)}}+ \frac{1}{M_0}+\frac{1}{\sqrt K_0}. \label{probe}
	\end{align} 
Combining \eqref{oute}--\eqref{probe} we complete the proof.

\section{Conclusion}

In this paper, we investigate an inverse potential scattering problem for the stochastic Schr\"odinger equation. The potential is assumed to be 
a generalized Gaussian random field whose covariance operator is a pseudo-differential operator. A  probabilistic stability is established 
for determining the principle of the covariance operator by multi-frequency far-field pattern. The analysis combines the ergodicity theory in stochastic analysis
and analytic continuation principle in complex theory. 
A possible continuation of this work is to study the stability in two dimensions. A more challenging problem is to investigate the inverse random scattering problems in inhomogeneous media. The method developed in this paper is not applicable to this case since
the Born series may not converge for large frequencies. We hope to be able to report the progresses on these problems in the future.

\end{document}